\documentclass[12pt]{article}
\usepackage{amsbsy}
\usepackage{amsmath,amsfonts,amssymb,amsthm,mathrsfs,mathtools,stmaryrd}
\usepackage{float,inputenc}
\usepackage{doi,url}
\usepackage[numbers]{natbib}
\usepackage{etoolbox,letltxmacro,xifthen,ifdraft} 
\usepackage[dvipsnames]{xcolor}
\usepackage{amsmath,amssymb,enumitem,bbm,mathrsfs,amsthm,aliascnt,mathtools}
\usepackage{graphicx,subcaption}
\usepackage{hyperref} 
\hypersetup{final,hidelinks}
\usepackage[atend]{bookmark} 

\providecommand{\email}[1]{\href{mailto:#1}{\nolinkurl{#1}}}

\setlist[enumerate,1]{label={(\roman*)}}
\setlist[enumerate,2]{label={(\alph*)}}
\setlist[enumerate,3]{label={(\Roman*)}}

\newcommand{\newsstheorem}[2]{
  \newaliascnt{#1}{dummy}
  \newtheorem{#1}[#1]{#2}
  \aliascntresetthe{#1}
  \expandafter\def\csname #1autorefname\endcsname{#2}
}
\numberwithin{dummy}{section}
\theoremstyle{plain}
  \newsstheorem{theorem}{Theorem}
  \newsstheorem{proposition}{Proposition}
  \newsstheorem{corollary}{Corollary}
  \newsstheorem{lemma}{Lemma}
  \newsstheorem{definition}{Definition}
  \theoremstyle{definition}
  \newsstheorem{example}{Example}
  \newsstheorem{assumption}{Assumption}
\theoremstyle{remark}
  \newsstheorem{remark}{Remark}

\newenvironment{eqnarr*}{\begin{IEEEeqnarray*}{rCl}}{\end{IEEEeqnarray*}\ignorespacesafterend}

\newcommand\RR{\mathbb{R}}
\newcommand\ZZ{\mathbb{Z}}

\newcommand\PP{\mathbb{P}}

\newcommand\EE{\mathbb{E}}
\newcommand\NN{\mathbb{N}}
\newcommand\CC{\mathbb{C}}

\newcommand\e{{\mathrm e}}

\newcommand\Ind{\mathbbm{1}}

\newcommand{\euler}[2]{{\displaystyle \left\langle {#1 \atop #2}\right\rangle}}

\makeatletter \newcommand\mathof[1]{{\operator@font#1}} \makeatother

\newcommand\iu{\mathof{i}}

\begin{document}

\title{ Counterbalancing steps at random \\ in a random walk}
\author{Jean Bertoin\footnote{ Institute of Mathematics, University of Zurich, Switzerland, \texttt{jean.bertoin@math.uzh.ch}}  }
\date{}
\maketitle 
\thispagestyle{empty}

\begin{abstract} 
{A random walk with counterbalanced steps is a process of partial sums $\check S(n)=\check X_1+ \cdots + \check X_n$ whose steps $\check X_n$ are given recursively as follows. For each $n\geq 2$, with a fixed probability $p$, $\check X_n$ is a new independent sample from some fixed law $\mu$, and with complementary probability $1-p$,  $\check X_n= -\check X_{v(n)}$ counterbalances a previous step, with $v(n)$ a uniform random pick from $\{1, \ldots, n-1\}$. We determine the asymptotic behavior of $\check S(n)$ in terms of $p$ and the first two moments of $\mu$. Our approach relies on a coupling with a reinforcement algorithm due to H.A. Simon, and on properties of random recursive trees and Eulerian numbers, which may be of independent interest. The method can be adapted to the situation where the step distribution $\mu$ belongs to the domain of attraction of a stable law.
 }
\newline  \vskip 1mm
{\normalfont \bfseries Keywords:}
Reinforcement, random walk, random recursive tree, Eulerian numbers, Yule-Simon model.\newline
\vskip 1mm
{\normalfont \bfseries Mathematics Subject Classification:}  60F05; 60G50;  05C05; 05A05.

\end{abstract}

\section{Introduction}\label{s:intro}

In short, the purpose of the present work is to investigate long time effects of an algorithm for counterbalancing steps in a random walk. As we shall first explain, our motivation stems from a nearest neighbor process on the integer lattice, known as the elephant random walk. 

The  elephant random walk is a stochastic process  with memory on  $\ZZ$, which records the trajectory of an elephant that makes steps with unit length  left or right at each positive integer time. It has been introduced by Sch\"utz and Trimper \cite{SchTr} and triggered a growing interest in the recent years; see, for instance, \cite{BaurBer, Bercu, ColGavSch1, ColGavSch2, KuMa, Kur}, and also \cite{Baur, BerLau2, Marco, Buart, GutStadt, Mitak} for related models.  The dynamics depend on a parameter $q\in[0,1]$ and can be described as follows. Let us assume that the first step of the elephant is a Rademacher variable, that is equals $+1$ or $-1$ with probability $1/2$.  For each time $n\geq 2$, the elephant remembers a step picked uniformly at random among those it made previously, and decides either to repeat it with probability $q$, or to make the opposite step with complementary probability $1-q$. Obviously, each step of the elephant has then  the Rademacher law, although the sequence of steps is clearly not stationary. 

Roughly speaking, it seems natural to generalize these dynamics and allow steps to have an arbitrary distribution on $\RR$, say $\mu$.  In this direction, K\"ursten \cite{Kur} pointed at an equivalent way of describing the dynamics of the elephant random walk which makes such generalization non trivial\footnote{ Note that merely replacing the Rademacher distribution for the first step of the elephant by $\mu$ would not be interesting, as one would then just get the evolution of the elephant random walk multiplied by some random factor with law $\mu$.}. Let $p\in[0,1]$, and imagine a walker who makes at each time a step which is either, with probability $p$, a new independent random variable with law $\mu$, or,  with probability $1-p$,  a repetition of one of his preceding steps picked uniformly at random. It is immediately checked that when $\mu$ is the Rademacher distribution,
then the walker follows the same dynamics as the elephant random walk with parameter $q=1-p/2$. When $\mu$ is an isotropic stable law, this is the model referred to as
the shark random swim by Businger \cite{Buart}, and more generally, when  $\mu$ is arbitrary, this is the step reinforced random walk that has been studied lately in e.g. \cite{Marco2, NRLP, Bescal, UNRBM}.

The model of K\"ursten yields an elephant random walk only with parameter $q\in[1/2,1]$; nonetheless  the remaining range can be obtained by a simple modification. Indeed, let again $p\in[0,1]$ and  imagine now an repentant walker who makes at each time a step which is either, with probability $p$,  a new independent random variable with law $\mu$, or,  with probability $1-p$, the \textit{opposite} of one of his previous steps picked uniformly at random.  When $\mu$ is the Rademacher distribution, we simply get
 the dynamics of the elephant random walk with parameter $q=p/2\in [0,1/2]$.

More formally, we consider a sequence $(X_n)$ of i.i.d. real random variables with some given law $\mu$ and a sequence  $(\varepsilon_n)_{n\geq 2}$ of i.i.d. Bernoulli variables with parameter $p\in[0,1]$, which we assume furthermore independent of $(X_n)$.
We construct a counterbalanced sequence $(\check X_n)$ by interpreting each $\{\varepsilon_n=0\}$ as a counterbalancing event and each  $\{\varepsilon_n=1\}$ as an innovation event. 
Specifically, we agree that $\varepsilon_1=1$ for definiteness and denote  the number of innovations after $n$ steps by
$${\mathrm i}(n) \coloneqq \sum_{j=1}^n\varepsilon_j\qquad\text{for }n\geq 1.$$
We introduce a sequence $(v(n))_{n\geq 2}$ of independent variables, where each $v(n)$ has the uniform distribution on $\{1, \ldots , n-1\}$, and which is also independent of $(X_n)$ and $(\varepsilon_n)$. We then define recursively
\begin{equation}\label{E:negreinf}
\check X_n \coloneqq\left\{
\begin{matrix}
- \check X_{v(n)}&\text{ if }\varepsilon_n=0, \\
X_{{\mathrm i}(n)}&\text{ if }\varepsilon_n=1.\\
\end{matrix}
\right.
\end{equation}
 Note that the same step can be counterbalanced several times, and also that certain steps counterbalance previous steps which in turn already counterbalanced earlier ones.  The process  
  $$\check S(n)\coloneqq \check X_1+ \cdots + \check X_n, \qquad n\geq 0$$ 
  which records the positions of the repentant walker as a function of time, is called here a random walk with counterbalanced steps.
 Note that for $p=1$, i.e. when no counterbalancing events occur, $\check S$ is just a usual random walk with i.i.d. steps.
 
 In short, we are interested in understanding how counterbalancing steps affects the asymptotic behavior of random walks.
 We first introduce some notation.  
 Recall that $\mu$ denotes the distribution of the first step $X_1=\check X_1$ and  write 
$$m_k\coloneqq \EE(X_1^k)=\int_\RR x^k \mu(\mathrm d x)$$ for the moment of order $k\geq 1$ of $X_1$ whenever  $X_1\in L^k(\PP)$.  
To start with, we point out that if  the first moment is finite, then the algorithm  \eqref{E:negreinf} yields 
the recursive equation
 $$\EE(\check S(n+1)) =pm_1 + (1-(1-p)/n) \EE(\check S(n)), \qquad n\geq 1,$$
with the initial condition $\EE(\check S(1))=m_1$. It follows easily that
 \begin{align*}\EE(\check S(n)) 
 &\sim \frac{p}{2-p}\, m_1 n\qquad \text{as }n\to \infty;
 \end{align*}
see e.g. Lemma 4.1.2  in \cite{Durrett}. Our first result about the ballistic behavior should therefore not come as a surprise.
\begin{proposition}\label{P0} Let $p\in[0,1]$. If  $X_1\in L^1(\PP)$, then  there is the convergence in probability 
$$\lim_{n\to \infty} \frac{\check S(n)}{n} = \frac{p}{2-p} m_1.$$
\end{proposition}

We see in particular that counterbalancing steps reduces the asymptotic velocity of a random walk by a factor $p/(2-p)<1$.  The velocity is smaller when the innovation rate $p$ is smaller
(i.e. when counterbalancing events have a higher frequency),
and vanishes as $p$ approaches $0+$. 

The main purpose of this work is to establish the asymptotic normality when $\mu$ has a finite second moment. 

\begin{theorem}\label{T2} Let $p\in(0,1]$.  If $X_1\in L^2(\PP)$, then there is the convergence in distribution
$$\lim_{n\to \infty} \frac{\check S(n)-\frac{p}{2-p}m_1n }{\sqrt n} = \mathcal N\left(0, \frac{m_2 -\left(\frac{p}{2-p}m_1\right)^2 }{3-2p} \right) ,$$
where the right-hand side denotes a centered Gaussian variable parametrized by mean and variance.
\end{theorem}
It is interesting to observe that the variance of the Gaussian limit depends linearly on the square $m_1^2$ of the first moment  and the second moment $m_2$ of $\mu$ only, although not just on the variance $m_2-m_1^2$ (except, of course, for $p=1$). 
Furthermore, it is not  always a monotonous function\footnote{  For instance, in the simplest case when $\mu$ is a Dirac point mass, i.e. $m_2=m_1^2$, then
 the variance is given by $\frac{4(1-p)m_2}{(3-2p)(2-p)^2}$ and reaches its maximum for 
$p=(9-\sqrt{17})/8\approx 0.6$. At the opposite, when $\mu$ is centered, i.e. $m_1=0$, the variance is given by $m_2/(3-2p)$ and  hence increases with $p$.} of the innovation rate $p$, and does not vanish when $p$ tends to $0$ either. 

Actually, our proofs of Proposition \ref{P0} and Theorem \ref{T2}  provide a much finer analysis than what is encapsulated  by those general statements.
Indeed, we shall  identify  the main actors for the evolution of $\check S$ and their respective contributions to its asymptotic behavior.
In short, we shall see that the ballistic behavior stems from those of the variables $X_j$ that have been used just once by the algorithm \eqref{E:negreinf}  (in particular, they have not yet been counterbalanced), whereas the impact of variables that occurred twice or more regardless of their signs $\pm$, is asymptotically negligible as far as only velocity is concerned. Asymptotic normality is more delicate to analyze. We shall show that, roughly speaking, it results from the combination of, on the one hand, the central limit theorem for certain centered random walks, and on the other hand, Gaussian fluctuations for the asymptotic frequencies of some pattern induced by \eqref{E:negreinf}.

Our analysis  relies on a natural coupling of the counterbalancing algorithm \eqref{E:negreinf}  with
a basic linear reinforcement
algorithm which has been  introduced a long time ago by H.A. Simon \cite{Simon} to explain the occurrence of certain heavy tailed distributions in a variety of empirical data. Specifically, Simon defined recursively a sequence denoted here by $(\hat X_n)$ (beware of the difference of notation between $\hat X$ and $\check X$) via
\begin{equation}\label{E:reinf}
\hat X_n \coloneqq \left\{
\begin{matrix}
\hat X_{v(n)}&\text{ if }\varepsilon_n=0, \\
X_{{\mathrm i}(n)}&\text{ if }\varepsilon_n=1.\\
\end{matrix}
\right.
\end{equation}
We stress that the same Bernoulli variables $\varepsilon_n$ and the same uniform variables  $v(n)$ are used to run both Simon's algorithm \eqref{E:reinf} and \eqref{E:negreinf}; in particular either $\check X_n= \hat X_n$ or $\check X_n= -\hat X_n$. 
It might then seem natural to refer to \eqref{E:negreinf} and \eqref{E:reinf} respectively as  negative and positive  reinforcement algorithms. However, in the literature, negative reinforcement usually refers to a somehow different notion  (see e.g. \cite{Pem}), and we shall thus avoid using this terminology.

A key observation is that  \eqref{E:negreinf}  can be recovered from \eqref{E:reinf} as follows. Simon's algorithm naturally encodes a genealogical forest with set of vertices $\NN=\{1, 2,\ldots,\}$ and edges $(v(j),j)$ for all $j\geq 2$ with $\varepsilon_j=0$; see the forthcoming Figure 1.
Then $\check X_n=\hat X_n$ if the vertex $n$  belongs to an even generation of its tree component, and $\check X_n=-\hat X_n$ if $n$ belongs to an odd generation.
On the other hand, the statistics of Simon's genealogical forest can be described in terms of independent random recursive trees (see e.g. Chapter 6 in \cite{Drmota} for background) conditionally on their sizes. 
This leads us to investigate the difference $\Delta(\mathbb T_k)$ between the number of vertices at even generations and the number of vertices at odd generations in a random recursive tree  $\mathbb T_k$ of size $k\geq 1$. The law of $\Delta(\mathbb T_k)$ can be expressed in terms of Eulerian numbers, and properties of the latter enable us either to compute explicitly or estimate certain quantities which are crucial for the proofs of Proposition \ref{P0} and Theorem \ref{T2}.  

It is interesting to compare asymptotic behaviors for counterbalanced steps  with those for reinforced steps. If we write $\hat S(n)=\hat X_1+ \cdots + \hat X_n$ for the random walk with  reinforced steps, then it is known that  the law of large numbers holds for
$\hat S$, namely $\hat S(n)/n\to m_1$ in $L^1$  when $\int_{\RR}|x|\mu(\mathrm{d}x)<\infty$, independently of the innovation parameter $p$. Further, regarding fluctuations when $\int_{\RR}|x|^2\mu(\mathrm{d}x)<\infty$, a phase transition occurs for the critical parameter $p_c=1/2$, in the sense that 
$\hat S$ is diffusive for $p>1/2$ and superdiffusive for $p<1/2$; see \cite{Bescal, UNRBM}. Despite of the natural coupling between \eqref{E:negreinf} and \eqref{E:reinf}, 
there are thus major differences\footnote{  This should not come as a surprise. In the simplest case when $\mu=\delta_1$ is the Dirac mass at $1$, one has $\hat S(n)\equiv n$ whereas $\check S$ is a truly stochastic process, even for $p=0$ when there is no innovation.} between the asymptotic behaviors of $\check S$ and of $\hat S$: 
Proposition \ref{P0} shows that the asymptotic speed of $\check S$ depends on $p$, and Theorem 
 \ref{T2}  that there is no such phase transition for counterbalanced steps and  $\check S$ is always diffusive. 
 
The phase transition for step reinforcement when $\mu$ has a finite second moment can be explained  informally as follows; for the sake of simplicity, suppose also that $\mu$ is centered, i.e. $m_1=0$.  There are $\mathrm i(n)\sim pn$ trees in Simon's genealogical forest, which are overwhelmingly microscopic (i.e. of size $O(1)$), whereas only a few trees reach the size $O(n^{1-p})$. Because $\mu$ is centered,  the contribution of microscopic trees to $\hat S(n)$ is of order $\sqrt n$, and that of the few largest trees of order $n^{1-p}$. This is the reason  why $\hat S(n)$ grows like $\sqrt n \gg n^{1-p}$ when $p>1/2$, and rather like $n^{1-p}\gg \sqrt n$ when $p<1/2$. For counterbalanced steps, we will see that, due to the counterbalancing mechanism, the contribution of a large tree with size $\ell\gg 1$ is now only of order $\sqrt \ell$. As a consequence, the contribution to $\check S(n)$ of the largest trees of Simon's genealogical forest  is only of order $O(n^{(1-p)/2})$. This is always much smaller than the contribution  of microscopic trees which remain of order $\sqrt n$. 
 We further stress that, even though only the sizes of the trees in Simon's genealogical forest are relevant for the analysis of the random walk $\hat S$ with  reinforced steps, the study of the random walk $\check S$ with counterbalanced steps is more complex and requires informations on the fine structure of those trees, not merely their sizes.

The rest of this text is organized as follows. Section 2 focusses on the purely counterbalanced case $p=0$. 
In this situation, for each fixed $n\geq 1$, the distribution of $\check S(n)$ can be expressed explicitly in terms of Eulerian numbers.
Section 3 is devoted to the coupling between the counterbalancing algorithm \eqref{E:negreinf} and H.A. Simon's algorithm \eqref{E:reinf}, and to the interpretation of the former in terms of a forest of random recursive trees induced by the latter. Proposition \ref{P0} and Theorem \ref{T2} are proved in Section 4, where we analyze more finely the respective contributions of some natural sub-families. Last, in Section 5, we present a stable version of Theorem \ref{T2} when $\mu$ belongs to the domain of attraction (without centering) of an $\alpha$-stable distribution for some $\alpha\in(0,2)$. 
 
\section{Warm-up: the purely counterbalanced case}
This section is devoted to the simpler situation\footnote{ Observe that this case without innovation has been excluded in Theorem \ref{T2}.} when $p=0$. So $\varepsilon_n\equiv 0$ for all $n\geq 2$, meaning that every step, except of course the first one, counterbalances some preceding step. The law $\mu$ then only plays a superficial role as it is merely relevant for the first step. For the sake of simplicity, we further focus on the case when $\mu=\delta_1$ is the Dirac mass at $1$.  

The dynamics are entirely encoded by the sequence $(v(n))_{n\geq 2}$ of independent uniform variables on $\{1,\ldots, n-1\}$; more precisely the purely counterbalanced sequence of bits is given by
\begin{equation}\label{E:p=0}
\check X_1= 1 \quad \text{and} \quad \check X_n=-\check X_{v(n)}\quad \text{for all }n\geq 2.
\end{equation}
The random algorithm \eqref{E:p=0} points at a convenient representation in terms of random recursive trees. Specifically, the sequence $(v(n))_{n\geq 2}$ encodes a random tree ${\mathbb T}_{\infty}$ with set of vertices $\NN$
 and set of edges $\{(v(n),n): n\geq 2\}$. Roughly speaking,  ${\mathbb T}_{\infty}$ is constructed recursively by incorporating vertices one after the other 
and creating an edge between each new vertex $n$ and its parent $v(n)$ which is picked uniformly at random in $\{1, \ldots, n-1\}$ and independently of the other vertices.
  If we view $1$ as the root of ${\mathbb T}_{\infty}$ and call a vertex $j$ \textit{odd} (respectively, \textit{even}) when its generation (i.e. its distance to the root in ${\mathbb T}_{\infty}$)
 is an odd (respectively, even) number, then 
 $$\check X_n= \left\{ \begin{matrix} 1 & \text{ if $n$ is an even vertex in $\mathbb T_{\infty}$,}\\
 -1 & \text{ if $n$ is an odd vertex in $\mathbb T_{\infty}$.}
 \end{matrix} \right.
 $$
 
 Let us now introduce some notation with that respect. For every $n\geq 1$, we write ${\mathbb T}_n$ for the restriction of ${\mathbb T}_{\infty}$
 to the set of vertices $\{1,\ldots, n\}$ and refer to  ${\mathbb T}_n$ as a random recursive tree with size $n$. We also write
 $\mathrm{Odd}({\mathbb T}_n)$ (respectively, $\mathrm{Even}({\mathbb T}_n)$) for the number  of odd (respectively, even) vertices in ${\mathbb T}_n$
 and set
  $$\Delta({\mathbb T}_n) \coloneqq \mathrm{Even}({\mathbb T}_n) - \mathrm{Odd}({\mathbb T}_n) = n-2 \mathrm{Odd}({\mathbb T}_n).$$
Of course, we can also express 
$$\Delta({\mathbb T}_n)= \check X_1+ \cdots + \check X_n,$$
which is the trajectory of an elephant full of regrets (i.e. for the parameter $q=0$).  

The main observation of this section is that  law of the number of odd vertices  is readily expressed in terms of the Eulerian numbers. Recall that
$\langle{}^n_k\rangle$ denotes the number of permutations $\varsigma$ of $\{1, \ldots,n\}$  with $k$ descents, i.e. such that $\#\{1\leq j<n: \varsigma(j) > \varsigma(j+1)\}=k$. Obviously $\langle{}^n_k\rangle\geq 1$ if and only if  $0\leq k <n$, 
and 
one has 
$$\sum_{k=0}^{n-1} \euler{n}{k}=n!.$$
The linear recurrence equation
\begin{equation}\label{E:recurEuler}
\euler{n}{k}=(n-k)\euler{{n-1}}{{k-1}}+(k+1)\euler{{n-1}}{k}
\end{equation}
is easily derived from a recursive construction of permutations  (see Theorem 1.3 in \cite{Peter}); we also mention the explicit formula (see Corollary 1.3 in \cite{Peter})
$$\euler{n}{k}=\sum _{j=0}^{k}(-1)^{j}{\binom {n+1}{j}}(k+1-j)^{n}.$$

\begin{lemma}\label{L1} For every $n\geq 1$, we have   
$$\PP(\mathrm{Odd}({\mathbb T}_n)=\ell)= \frac{1}{(n-1)! } \euler{{n-1}}{{\ell-1}},$$
with the convention\footnote{ Note that this convention is in agreement  with the linear recurrence equation \eqref{E:recurEuler}.} that $\langle{}^{\ 0}_{-1}\rangle=1$ in the right-hand side for $n=1$ and $\ell=0$.
\end{lemma} 
\begin{proof}   Consider $n\geq 1$ and note from the very construction of random recursive trees that there is the identity
$$  \PP(\mathrm{Odd}({\mathbb T}_{n+1})=\ell) = \frac{\ell}{n}\, \PP(\mathrm{Odd}({\mathbb T}_n)=\ell) +  \frac{n+1-\ell}{n}\, \PP(\mathrm{Odd}({\mathbb T}_n)=\ell-1).$$
Indeed,  the first term of the sum in the right-hand side accounts for the event that the parent $v(n+1)$ of the new vertex $n+1$ is an odd vertex (then $n+1$ is an even vertex), and the second term for the event that  $v(n+1)$ is an even vertex (then $n+1$ is an odd vertex). 

In terms of  $A(n,k) \coloneqq n! \PP(\mathrm{Odd}({\mathbb T}_{n+1})=k+1)$, this yields 
$$ A(n,k) = (k+1) A(n-1,k) + (n-k) A(n-1,k-1),$$
which is the linear recurrence equation \eqref{E:recurEuler} satisfied by the Eulerian numbers. Since plainly $A(1,0)=\PP(\mathrm{Odd}(2)=1)=1=\langle{}^1_0\rangle$,
we conclude by iteration that $A(n,k)=\langle{}^n_k\rangle$ for all $n\geq 1$ and $0\leq k<n$. Last, the formula in the statement also holds for $n=1$ since $\mathrm{Odd}(1)=0$. 
\end{proof}
\begin{remark} Lemma \ref{L1} is implicit in  Mahmoud \cite{Mahmoud}\footnote{ Beware however that the definition of Eulerian numbers in \cite{Mahmoud} slightly differs from ours, namely 
$\langle{}^n_k\rangle$ there corresponds to $\langle{}^{\ n}_{k-1}\rangle$ here.}. Indeed $\mathrm{Odd}({\mathbb T}_n)$ can be viewed as the number of blue balls in an analytic Friedman's urn model
started with one white ball and replacement scheme $\left({}^{0\ 1}_{1\ 0}\right)$; see Section 7.2.2 in \cite{Mahmoud}. In this setting, Lemma \ref{L1} is equivalent to 
the formula for the number of white balls at the bottom of page 127 in \cite{Mahmoud}. Mahmoud relied on the analysis of the differential system associated to the replacement scheme via a Riccati differential equation and inversion of generating functions. The present approach based on the linear recurrence equation \eqref{E:recurEuler} is more direct. 
Lemma \ref{L1} is also a closed relative to a result due to Najock and Heyde \cite{NaHeyde} (see also Theorem 8.6 in Mahmoud \cite{Mahmoud} and Section 6.2.4 in Drmota \cite{Drmota}) which states that  the number of leaves in  a random recursive tree with size $n$ has the same distribution as that  appearing in Lemma \ref{L1}. 
\end{remark}
We next point at a useful identity related to Lemma \ref{L1} which goes back to Laplace (see Exercise 51 of Chapter I in Stanley \cite{Stanley}) and is often attributed to Tanny \cite{Tanny}.   For every $n\geq 0$, there is the identity in distribution
\begin{equation}\label{E:tanny}
\mathrm{Odd}({\mathbb T}_{n+1})\,{\overset{\mathrm{(d)}}{=}}\, \lceil U_1+\cdots + U_n\rceil,
\end{equation}
where in the right-hand side, $U_1, U_2, \ldots$ is a sequence of i.i.d. uniform variables on $[0,1]$ and $\lceil \cdot \rceil$ denotes  the ceiling function.
We now record for future use the following consequences.
\begin{corollary}\label{C0}
\begin{itemize}
\item[(i)] For every $n\geq 2$, the variable $\Delta({\mathbb T}_n)$ is symmetric, i.e. $\Delta({\mathbb T}_n)\,{\overset{\mathrm{(d)}}{=}} \, -\Delta({\mathbb T}_n)$,
and in particular $\EE(\Delta({\mathbb T}_n))= 0$.
\item[(ii)] For all $n\geq 3$, one has $ \EE(\Delta({\mathbb T}_n)^2)= n/3$.
\item[(iii)]  For all $n\geq 1$, one has $\EE(|\Delta({\mathbb T}_n)|^4) \leq 6n^2$.
\end{itemize}
\end{corollary}

\begin{proof} (i) Equivalently, the assertion claims that in a random recursive tree of size at least $2$, the number of odd vertices and the number of even vertices have the same distribution. This is immediate from \eqref{E:tanny} and can also be checked directly from the construction.

(ii) This has been already observed by Sch\"utz and Trimper \cite{SchTr} in the setting of the elephant random walk; for the sake of completeness we present a short argument.
The vertex $n+1$ is odd (respectively, even) in ${\mathbb T}_{n+1}$ if and only if its parent is an even (respectively, odd) vertex in ${\mathbb T}_{n}$. Hence one has
$$\EE(\Delta({\mathbb T}_{n+1})-\Delta({\mathbb T}_{n})\mid {\mathbb T}_{n})= -\frac{1}{n}\Delta({\mathbb T}_{n}),$$
and since $\Delta({\mathbb T}_{n+1})-\Delta({\mathbb T}_{n})= \pm 1$, this yields the recursive equation 
$$\EE(\Delta({\mathbb T}_{n+1})^2)= (1-2/n)\EE(\Delta({\mathbb T}_{n})^2) +1.$$
By iteration, we conclude that $ \EE(\Delta({\mathbb T}_n)^2)= n/3$ for all $n\geq 3$.

(iii)
Recall that the process of the fractional parts $\{U_1+\cdots + U_n\}$ is a Markov chain on $[0,1)$ whose distribution at any fixed time $n\geq 1$
is uniform on $[0,1)$. Writing $V_n= 1-2U_n$ and $W_n=2\{U_1+\cdots + U_n\}-1$, we see that $V_1, V_2, \ldots$  is a sequence of i.i.d. uniform variables on $[-1,1]$
and  that $W_n$ has the uniform distribution on $[-1,1]$ too. 

The characteristic function of the uniform variable $V_j$  is 
$$\EE(\exp(\iu \theta V_j))=\theta^{-1}\sin(\theta) = 1-\frac{\theta^2}{6} + \frac{\theta^4}{120} + O(\theta^6)\qquad\text{as }\theta \to 0,$$
and therefore for every $n\geq 1$, 
\begin{align*}\EE(\exp(\iu \theta (V_1+\cdots+V_n)))&=\left(1-\frac{\theta^2}{6} + \frac{\theta^4}{120} + O(\theta^6)\right)^n\\
&=1-\frac{n}{6}\theta^2 + \left(\frac{n}{120}+ \frac{n(n-1)}{72}\right) \theta^4+ O(\theta^6). 
\end{align*}
It follows that 
$$\EE( (V_1+\cdots+V_n)^4)= 24 \left(\frac{n}{120}+ \frac{n(n-1)}{72}\right) \leq n^2/3.$$

We can rephrase \eqref{E:tanny} as the identity in distribution
$$\Delta({\mathbb T}_{n+1})\,{\overset{\mathrm{(d)}}{=}}\, V_1+\cdots+V_n + W_n.$$ 
Since $\EE(W_n^4)=1/3$, the proof is completed with the elementary bound $(a+b)^4\leq 16(a^4+b^4)$.
\end{proof}

We now conclude this section with an application of \eqref{E:tanny} to the asymptotic normality of $\Delta({\mathbb T}_n)$. Since $\EE(U)=1/2$ and $\mathrm{Var}(U)=1/12$, the classical central limit theorem immediately yields the following.
 \begin{corollary}\label{C1} Assume $p=0$ and $\mu=\delta_1$. One has
$$\lim_{n\to \infty} \frac{ \Delta({\mathbb T}_n)}{\sqrt n} = {\mathcal N}(0,1/3)\qquad \text{in distribution.}$$
 \end{corollary}
Corollary \ref{C1} goes back to  \cite{NaHeyde} in the setting of the number of leaves in random recursive trees; see also \cite{BaurBer, Bercu, ColGavSch1, ColGavSch2} for alternative proofs in the framework of the elephant random walk. 

\section{Genealogical trees in Simon's algorithm}
From now on,  $\mu$ is an arbitrary probability law on $\RR$ and we also suppose that the innovation rate is strictly positive,  $p\in(0,1)$.
Recall the construction of the sequence $(\hat X_n)$ from Simon's reinforcement algorithm \eqref{E:reinf}. 
Simon was interested in the asymptotic frequencies of variables having a given number of occurrences. Specifically, for every $n,j\in \NN$, we write  
$${N}_j(n)\coloneqq \#\{\ell \leq n: \hat X_{\ell}=X_j\}$$ 
for the number of occurrences of the variable $X_j$ until the $n$-th step of the algorithm \eqref{E:reinf}, and 
\begin{equation}\label{E:nudef}
\nu_k(n)\coloneqq \#\{1\leq j \leq \mathrm{i}(n): {N}_j(n)=k\}, \qquad k\in\NN
\end{equation}
for the  number of such variables  that have occurred exactly $k$ times. Observe also that the number of innovations satisfies the law of large numbers ${\mathrm i}(n)\sim p n$ a.s.

  \begin{lemma}\label{L2} For every 
$k\geq 1$, we have
$$\lim_{n\to \infty} \frac{\nu_k(n)}{pn}=
\frac{1}{1-p}
 {\mathrm B}(k,1+1/(1-p)) \qquad\text{in probability,}$$
where ${\mathrm B}$ denotes the beta function.
\end{lemma}

Lemma \ref{L2} is essentially due to H.A. Simon \cite{Simon}, who actually only established the convergence of the mean value.
The strengthening to convergence in probability can be obtained  as in \cite{BRST} from a concentration argument based on the Azuma-Hoeffding's inequality;
see Section 3.1  in \cite{PPS}.  The right-hand side in the formula is a probability mass on $\NN$ known as the Yule-Simon distribution with parameter $1/(1-p)$.
We record for future use a couple of  identities which are easily checked from the integral definition of the beta function:
\begin{equation} \label{E:sumbeta} 
\frac{1}{1-p} \sum_{k=1}^{\infty} {\mathrm B}(k,1+1/(1-p)) =1
\end{equation}
and 
\begin{equation} \label{E:sumkbeta} 
 \frac{1}{1-p} \sum_{k=1}^{\infty} k {\mathrm B}(k,1+1/(1-p)) = \frac{1}{p}.
\end{equation}

For $k=1$,  Lemma  \ref{L2} reads 
\begin{equation}\label{E:Simon1}
\lim_{n\to \infty} n^{-1}\nu_1(n)= \frac{p}{2-p}\qquad\text{in probability.}
\end{equation}
We shall also need to estimate the fluctuations, which can be derived by specializing a Gaussian limit theorem for extended P\'olya urns due to Bai \textit{et al.} \cite{BHZ}. 

\begin{lemma}\label{L5} There is the convergence in distribution
$$\lim_{n\to \infty} \frac{\nu_1(n)- np/(2-p)}{\sqrt n} = \mathcal N\left(0,\frac{2p^3-8p^2+6p}{(3-2p)(2-p)^2} \right).$$
\end{lemma}

\begin{proof} The proof relies on the observation that Simon's algorithm can be coupled with a two-color urn governed by the same sequences of random bits $(\varepsilon_n)$ and of uniform variables $(v(n))$ as follows. Imagine that we observe the outcome of Simon's algorithm at the $n$-step
and that for each $1\leq j \leq n$, we associate a white ball if the variable $\hat X_j$  appears exactly once, and a red ball otherwise.
At the initial time $n=1$, the urn contains just one white ball and no red balls. 
At each step $n\geq 2$, a ball picked uniformly at random in the urn (in terms of Simon's algorithm, this is given by the uniform variable $v(n)$). 
If $\varepsilon_n=1$, then the ball picked is returned to the urn and one adds a white ball (in terms of Simon's algorithm, this corresponds to an innovation 
and $\nu_1(n)=\nu_1(n-1)+1$). 
If $\varepsilon_n=0$, then the ball picked is removed from the urn and one adds two red balls
(in terms of Simon's algorithm, this corresponds to a repetition  
and either $\nu_1(n)=\nu_1(n-1)-1$ if the ball picked is white, or $\nu_1(n)=\nu_1(n-1)$ if the ball picked is red). 
By construction, the number $W_n$ of white  balls in the urn coincides with the number $\nu_1(n)$ of variables that have appeared exactly once 
in Simon's algorithm \eqref{E:reinf}. 

We shall now check our claim in the setting of \cite{BHZ} by specifying the quantities which appear there. The evolution of number of white balls in the urn is governed by Equation (2.1) in \cite{BHZ}, viz. 
$$W_n=W_{n-1}+ I_nA_n+(1-I_n)C_n,$$ where $I_n=1$ if a white ball is picked and $I_n=0$ otherwise. In our framework, we further have 
$A_n=2\varepsilon_n-1$ and $C_n=\varepsilon_n$. If we write ${\mathcal F}_n$ for the natural filtration generated by the variables $(A_k,C_k, I_k)_{k\leq n}$,
then $A_n$ and $C_n$ are independent of ${\mathcal F}_{n-1}$ with
$$\EE(A_n)= 2p-1, \quad \EE(C_n)= p, \quad \mathrm{Var}(A_n)= 4(p-p^2), \quad \mathrm{Var}(C_n)= p-p^2.$$
This gives in the notation (2.2) of \cite{BHZ}:
\begin{align*}
\sigma^2_M&= \frac{p}{2-p}4(p-p^2)+ \left(1-\frac{p}{2-p}\right) (p-p^2) + (p-1)^2 \frac{p}{2-p} \left(1-\frac{p}{2-p}\right)\\
&= \frac{2p^3-8p^2+6p}{(2-p)^2},
\end{align*}
and finally 
$$\sigma^2=\frac{2p^3-8p^2+6p}{(3-2p)(2-p)^2}.$$
Our claim can now be seen as a special case of Corollary 2.1 in \cite{BHZ}. 
\end{proof}

We shall also need a refinement of Lemma \ref{L2} in which one does not only record the number of occurrences of the variable $X_j$, but more generally the genealogical structure of these occurrences. We need to introduce first some notation in that respect. 

Fix $n\geq 1$ and let some $1\leq j \leq {\mathrm i}(n)$ (i.e. the variable $X_j$ has already appeared at the $n$-th step of the algorithm). 
Write   $\ell_1< \ldots < \ell_k \leq n$ for the increasing sequence of steps of the algorithm at which $X_j$ appears, where $k={N}_j(n)\geq 1$.
The genealogy of occurrences of the variable $X_j$ until the $n$-th step is recorded as a tree $T_j(n)$ on $\{1, \ldots, k\}$
such that for every $1\leq a < b \leq k$, $(a,b)$ is an edge of  $T_j(n)$ if and only if $v(\ell_b)=\ell_a$, that is, if and only if the identity 
$\hat X_{\ell_b}=X_j$ actually results from the fact that  the algorithm repeats the variable $\hat X_{\ell_a}$ at its $\ell_b$-th step. 
Plainly, $T_j(n)$ is an increasing tree with size $k$, meaning a tree on $\{1, \ldots, k\}$ such that the sequence of vertices along any branch from the 
root $1$ to a leaf is increasing. In this direction, we recall that there are $(k-1)!$ increasing trees with size $k$ and that the uniform distribution of the set 
increasing trees with size $k$ coincides with the law of the random recursive tree of size $k$, ${\mathbb T}_k$. See for instance Section 1.3.1 in Drmota \cite{Drmota}. 

\begin{figure}
\begin{center}
\includegraphics[height=6cm]{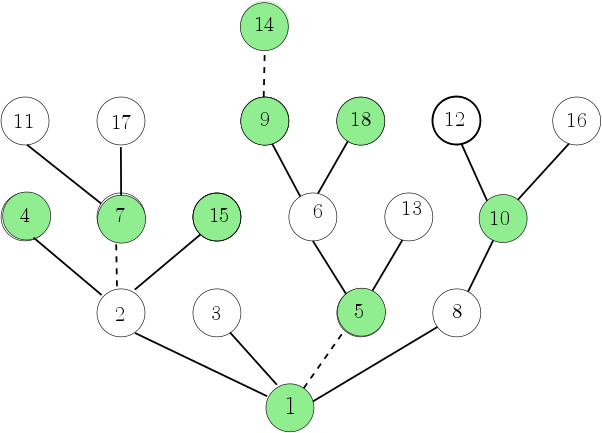}
\end{center}
\caption{ Example of a genealogical forest representation of Simon's algorithm \eqref{E:reinf} after 18 steps. The dotted edges account for innovation events, i.e. $\varepsilon_j=1$ 
and the four genealogical trees are rooted at $1,5,7,14$. In each subtree, vertices at even generations are colored in green and vertices at odd generation  in white. For instance the genealogical tree $T_2(18)$ is rooted at $5$, it has $3$ even vertices and $2$ odd vertices. 
 }\label{fig:summary}
\end{figure}
More generally, the distribution of the entire genealogical forest  given the sizes of the genealogical trees can be described as follows.
\begin{lemma}\label{L3} Fix $n\geq 1$, $1\leq k\leq n$, and let $n_1, \ldots , n_k\geq 1$ with sum $n_1+ \cdots + n_k=n$. Then conditionally on
${N}_j(n)=n_j$ for every $j=1, \ldots, k$, the genealogical trees $T_1(n), \ldots , T_k(n)$ are independent random recursive trees with respective sizes $n_1, \ldots , n_k$.
\end{lemma}
\begin{proof} Recall that the set $\{(v(j),j)$ for $1\leq j\leq n\}$ is that of the edges of the random recursive tree  with size $n$, ${\mathbb T}_n$.
The well-known splitting property states that removing a given edge, say $(v(j),j)$ for some fixed $j$, from ${\mathbb T}_n$ produces two subtrees which in turn, conditionally on their sizes, are two independent random recursive trees. This has been observed first by Meir and Moon \cite{MeirMoon}; see also \cite{BBsurvey} and references therein for more about this property.

The genealogical trees $T_1(n), \ldots , T_k(n)$ result by removing the edges $(v(j),j)$ in ${\mathbb T}_n$ for which $\varepsilon_j=1$ and enumerating in each subtree component their vertices in the increasing order. 
Our statement is now easily seen by applying iteratively this  splitting property. \end{proof}

We shall also need for the proofs of Proposition \ref{P0} and Theorem \ref{T2} an argument of uniform integrability that relies in turn on the following lemma. 
Recall that if $T$ is a rooted tree, $\Delta(T)$ denotes the difference between the number of vertices at even distance from root and that at odd distance. 
\begin{lemma}\label{L6} For every $1<\beta<2\wedge \frac{1}{1-p}$, one has
$$\sup_{n\geq 1} \frac{1}{n} \sum_{j=1}^{n}\EE({N}_j(n)^{\beta})< \infty$$
and
$$\sup_{n\geq 1} \frac{1}{n}\EE\left( \sum_{j=1}^{{\mathrm i}(n)} |\Delta(T_j(n) ) |^{2\beta} \right) <\infty.$$
\end{lemma}
\begin{proof} The first claim is a consequence of Lemma 3.6 of \cite{NRLP} which states that for $\beta \in(1,1/(1-p))$ [beware that the parameter denoted by $p$ in \cite{NRLP} is actually $1-p$ here], there exists numerical constants $c>0$ and $\eta\in(0,1)$ such that 
$\EE({N}_j(n)^{\beta})\leq c(n/j)^{\eta}$ for all $1\leq j \leq n$.

Next, combining  Jensen's inequality with Corollary \ref{C0}(iii), we get that for $k\geq 2$
$$\EE( |\Delta({\mathbb T}_k)|^{2\beta}) \leq \EE(|\Delta({\mathbb T}_k)|^4)^{\beta/2}\leq 6k^{\beta}.$$
Then recall that conditionally on ${N}_j(n)=k\geq 1$, $T_j(n) $ has the  law  of   the random recursive tree with size $k$,  ${\mathbb T}_k$, and hence 
\begin{align*}
\EE\left( \sum_{j=1}^{{\mathrm i}(n)} |\Delta(T_j(n) ) |^{2\beta} \right) 
&=  \sum_{j=1}^{n} \left(\sum_{k=1}^{n} \EE( |\Delta({\mathbb T}_k)|^{2\beta}) \PP({N}_j(n)=k)\right)\\
&\leq   6\sum_{j=1}^{n} \left(\sum_{k=1}^{n} k^{\beta} \PP({N}_j(n)=k)\right).
\end{align*}
We know from the first part that this last quantity is finite, and the proof is complete.
\end{proof}

\section{Proofs of the main results}

As its title indicates, the purpose of this section is to establish Proposition \ref{P0} and  Theorem  \ref{T2}. 
The observation that for every $n\geq 1$ and $1\leq j \leq {{\mathrm i}(n)}$, the variable $X_j$ appears exactly 
$\mathrm{Even}(T_j(n))$ times and its opposite $-X_j$ exactly $\mathrm{Odd}(T_j(n))$ times until the $n$-step of  the algorithm \eqref{E:negreinf},
yields the identity
\begin{equation}\label{E:repres}
\check S(n)\coloneqq \sum_{i=1}^n\check X_i = \sum_{j=1}^{{\mathrm i}(n)} \Delta(T_j(n)) X_j,
\end{equation}
which lies at the heart of our approach. We stress that in \eqref{E:repres} as well as in related expressions that we shall use in the sequel,  the sequence of i.i.d. variables $(X_n)$ and the family of genealogical trees $(T_j(n))$ are independent, because the latter are constructed from the sequences $(\varepsilon_n)$ and $(v(n))$ only.

Actually, our proof analyzes more precisely the effects of the counterbalancing algorithm \eqref{E:negreinf}  by estimating specifically the contributions of certain sub-families  to the asymptotic behavior of  $\check S$. Specifically, we set for every $k\geq 1$, 
\begin{equation} \label{E:checkSk}
\check S_{k}(n) \coloneqq \sum_{j=1}^{{\mathrm i}(n)}  \Delta(T_j(n)) X_j \Ind_{N_j(n)=k},
\end{equation}
so that
$$\check S(n) = \sum_{k=1}^{n} \check S_k(n).$$

\subsection{Proof of Proposition \ref{P0}}
The case $p=1$ (no counterbalancing events) of Proposition \ref{P0} is just the weak law of large numbers, and the case $p=0$ (no innovations) is a consequence of Corollary \ref{C1}. The case $p\in(0,1)$ derives from the next lemma which shows more precisely  that the variables $X_j$ that have appeared 
in the algorithm \eqref{E:negreinf} but have not yet counterbalanced determine 
 the ballistic behavior of $\check S$, whereas those  that have appeared twice or more (i.e. such that $N_j(n)\geq 2$)  have a negligible impact.

\begin{lemma}\label{L7}
Assume that $X_1\in L^1(\PP)$ and recall that $m_1=\EE(X_1)$. Then the following limits hold in probability:
 \begin{enumerate}
\item[(i)] $\lim_{n\to \infty} n^{-1} \check S_1(n) =m_1p/(2-p)$,
\item[(ii)]
$\lim_{n\to \infty} n^{-1} \sum_{k=2}^n \left| \check S_{k}(n) \right| = 0$. 
\end{enumerate}
\end{lemma}
\begin{proof}
(i) Recall the notation \eqref{E:nudef} and that the sequence of i.i.d. variables $(X_j)$ is independent of the events $\{{N}_j(n)=1\}$.
We see that there is the identity in distribution 
$$\check S_1(n) \,{\overset{\mathrm{(d)}}{=}}\, S_1(\nu_1(n)),$$
where $S_1(n)=X_1+\cdots + X_n$ is the usual random walk.
The claim follows readily from the law of large numbers and \eqref{E:Simon1}.

(ii)
We first argue that for each fixed $k\geq 2$,
\begin{equation}\label{E:fixedk}
\lim_{n\to \infty} n^{-1} \check S_{k}(n) = 0 \qquad\text{almost surely.}
\end{equation}
Indeed, recall that $\nu_k(n)$ denotes the  number of genealogical trees $T_j(n)$ with size $k$.
 It follows from Lemma \ref{L3} that conditionally on $\nu_k(n)=\ell$, the sub-family of such $T_j(n)$ enumerated in the increasing order of the index $j$, is given by $\ell$ i.i.d. copies of the random recursive tree ${\mathbb T}_k$.  Hence, still conditionally on $\nu_k(n)=\ell$,  enumerating the elements of the sub-family  $\{X_j   \Delta(T_j(n)): {N}_j(n)=k\}$  in the increasing order of $j$ yields $\ell$ independent variables, each being distributed as $X_1 \Delta({\mathbb T}_k)$ with $X_1$ and $ \Delta({\mathbb T}_k)$ independent. We deduce from Corollary \ref{C0}(i) that the variable $X_1 \Delta({\mathbb T}_k)$ symmetric, and since it is also integrable,  it is centered. Since $\nu_k(n)\leq n$, this readily entails \eqref{E:fixedk} by an application of the law of large numbers.  

The proof can be completed by an argument of uniform integrability. In this direction, fix an arbitrarily large integer $\ell$ and write by the triangular inequality
\begin{align*} \frac{1}{n} \EE\left( \sum_{k=\ell}^n \left| \check S_{k}(n) \right| \right) &\leq  \frac{1}{n} \EE\left( \sum_{j=1}^{{\mathrm i}(n)} |X_j |  {N}_j(n) \Ind_{{N}_j(n)\geq \ell} \right)\\
&=\frac{ \EE(|X_1|)}{n} \sum_{j=1}^{n} \EE\left(  {N}_j(n) \Ind_{{N}_j(n)\geq \ell} \right)\\
&\leq  \ell^{1-\beta} \frac{ \EE(|X_1|)}{n}\sum_{j=1}^{n} \EE\left(  {N}_j(n)^{\beta}  \right),
\end{align*}
where the last inequality holds for any $\beta>1$. We see from Lemma \ref{L6} that the right-hand side converges to $0$ as $\ell \to \infty$ uniformly in $n\geq 1$, and the rest of the proof is straightforward.
\end{proof}

\subsection{Proof of Theorem \ref{T2}}
For $p=1$ (no counterbalancing events), Theorem \ref{T2}  just reduces to the classical central limit theorem, so we assume $p\in(0,1)$.
The first step of the proof consists in determining jointly the fluctuations of the components $\check S_k$ defined in \eqref{E:checkSk}.

\begin{lemma}\label{L8} Assume that $m_2=\EE(X^2_1)<\infty$. 
Then as $n\to \infty$,  the sequences of random variables
$$ \frac{ \check S_1(n) -pm_1/(2-p)}{\sqrt n} \qquad \text{(for $k=1$)}$$
and 
$$ \frac{ \check S_{k}(n)}{\sqrt n}\qquad \text{(for $k\geq 2$)}$$
converge jointly in distribution 
 towards a sequence 
 $$\left( \mathcal N_{k}(0,\sigma^2_{k})\right)_{k\geq 1}$$
of independent centered Gaussian variables, where 
$$ \sigma_1^2\coloneqq \frac{p m_2}{2-p} - \frac{p^2m_1^2}{(3-2p)(2-p)^2},$$ $\sigma^2_2\coloneqq0$, and 
 $$\sigma^2_{k}\coloneqq  \frac{ k p  m_2}{3(1-p)}   \mathrm B(k,1+1/(1-p)) \qquad \text{for $k\geq 3$}.$$
 \end{lemma} 
\begin{proof}   For each $k\geq 1$, let $(Y_k(n))_{n\geq 1}$ be a sequence of i.i.d. copies of  $\Delta(\mathbb T_k) X$, where $X$ has the law $\mu$ and is independent of the random recursive tree $\mathbb T_k$. We further assume that these sequences are independent. Taking partial sums yields a sequence indexed by $k$ of independent random walks
$$S_k(n)= Y_k(1)+ \cdots + Y_k(n), \qquad n\geq 0. $$

For each $n\geq 1$, the family of  blocks
$$B_{k}(n)\coloneqq \{j\leq {\mathrm i}(n): N_j(n)=k\}  \qquad \text{for }1\leq k\leq {\mathrm i}(n)$$
forms a random partition of $\{1, \ldots, {\mathrm i}(n)\}$ 
which is independent of the $X_j$'s.
Recall that we are using the notation $\nu_k(n)=\#B_k(n)$, and also from Lemma \ref{L3}, that conditionally on the $N_j(n)$'s, the genealogical trees $T_j(n)$ are independent random recursive trees.  
We now see from the very definition \eqref{E:checkSk} that for every fixed $n\geq 1$,  there is the identity in distribution
$$\left( \check S_k(n)\right)_{k\geq 1} \,{\overset{\mathrm{(d)}}{=}}\, \left(  S_k(\nu_k(n))\right)_{k\geq 1},$$
where in the right-hand side, the random walks $(S_k)_{k\geq 1}$ are independent of the sequence of block sizes  $(\nu_k(n))_{k\geq 1}$.

Next we write, first for $k=1$,
$$S_1(\nu_1(n)) - \frac{pn}{2-p} m_1 = S_1\left(\left\lfloor \frac{pn}{2-p}\right \rfloor \right) - \frac{pn}{2-p}m_1 + \sum_{j=\lceil pn/(2-p)\rceil}^{\nu_1(n)}  Y_1(j) ,$$
second $S_2\equiv 0$ (since $\Delta(\mathbb T_2)\equiv 0$) for $k=2$, and third, for $k\geq 3$,
$$S_k(\nu_k(n)) = S_k\left(\left\lfloor \frac{pn}{1-p}
 {\mathrm B}(k,1+1/(1-p))\right \rfloor \right)  + \sum_{j=\lceil \frac{pn}{1-p}
 {\mathrm B}(k,1+1/(1-p))\rceil}^{\nu_k(n)}  Y_k(j) ,$$
with the usual summation convention that $\sum_{j=a}^b = - \sum_{j=b}^a$ when $b<a$.

Since the i.i.d. variables  $Y_1(\cdot)$ have mean $m_1$ and variance $m_2-m_1^2$,  the central limit theorem ensures that there is the convergence in distribution
\begin{equation} \label{E:clt}
\lim_{n\to \infty} n^{-1/2}\left( S_1\left(\left\lfloor \frac{pn}{2-p}\right \rfloor \right) - \frac{pn}{2-p}m_1\right) = \mathcal N_1\left( 0,\frac{p (m_2-m_1^2)}{2-p}\right ).
\end{equation}
Similarly, for $k\geq 3$,  each $Y_k(n)$ is centered with variance $km_2/3$ (by Corollary \ref{C0}(i-ii)) and hence, using the notation in the statement,  there is the convergence in distribution
\begin{align} \label{E:clt3}
\lim_{n\to \infty} n^{-1/2}S_k\left(\left\lfloor \frac{pn}{1-p}
 {\mathrm B}(k,1+1/(1-p))\right \rfloor \right) & = \mathcal N_k\left( 0,\sigma_k^2 \right)
\end{align}
Plainly, the weak convergences \eqref{E:clt} and \eqref{E:clt3} hold jointly when we agree that the limits are independent Gaussian variables.

Next, from Lemma \ref{L2} and the fact that for $k\geq 3$, the i.i.d. variables $Y_k(j)$ are centered with finite variance, we easily get
$$\lim_{n\to \infty} n^{-1/2}  \left | \sum_{j=\lceil \frac{pn}{1-p}
 {\mathrm B}(k,1+1/(1-p))\rceil}^{\nu_k(n)}  Y_k(j) \right |=0  \qquad \text{in }L^2(\PP).$$
Finally, for $k=1$, we write
$$ \sum_{j=\lceil pn/(2-p)\rceil}^{\nu_1(n)}  Y_1(j)=  m_1(\nu_1(n) - \lfloor pn/(2-p)\rfloor) + \sum_{j=\lceil pn/(2-p)\rceil}^{\nu_1(n)}  (Y_1(j)-m_1).$$
On the one hand, we have from the same argument as above that
$$\lim_{n\to \infty} n^{-1/2}  \left |  \sum_{j=\lceil pn/(2-p)\rceil}^{\nu_1(n)}  (Y_1(j)-m_1)\right |=0  \qquad \text{in }L^2(\PP).$$
On the other hand, we already know from Lemma \ref{L5} that there is  the convergence in distribution
$$\lim_{n\to \infty}  m_1\frac{\nu_1(n) - \lfloor pn/(2-p)\rfloor}{\sqrt n} =  \mathcal N\left( 0,\frac{2p^3-8p^2+6p}{(3-2p)(2-p)^2} m_1^2\right).$$
Obviously,  this convergence in law  hold jointly with \eqref{E:clt} and \eqref{E:clt3}, where the limiting Gaussian variables are independent. Putting the pieces together, this completes the proof.
\end{proof}

The final step for the proof of Theorem \ref{T2} is the following lemma.
\begin{lemma}\label{L9} We have
$$\lim_{K\to \infty} \sup_{n\geq 1} n^{-1} \EE\left( \left| \sum_{k\geq K}  \check S_{k}(n)\right|^2 \right) =0.$$
\end{lemma}

\begin{proof} We write
$$\sum_{k\geq K}  \check S_k(n)= \sum_{j=1}^{n} X_j   \Delta(T_j(n)) \Ind_{{N}_j(n)\geq K}.$$
Since   the  $X_j$ are independent of the $T_j(n)$, we get
\begin{align*} \EE\left( \left | \sum_{k\geq K} \check S_k(n)\right|^2 \right)&= \EE\left(  \sum_{j,{j'}=1}^{n} X_j X_{j'}  \Delta(T_j(n)) \Ind_{{N}_{j}(n)\geq K}  \Delta(T_{j'}(n)) \Ind_{{N}_{j'}(n)\geq K}\right)\\
&\leq m_2  \sum_{j,{j'}=1}^{n}   \EE\left(  \Delta(T_j(n)) \Ind_{{N}_{j}(n)\geq K}  \Delta(T_{j'}(n)) \Ind_{{N}_{j'}(n)\geq K}\right).
\end{align*}
We evaluate the expectation in the right-hand side by conditioning first on  $N_j(n)=k$ and $N_{j'}(n)=k'$ with $k,k'\geq 3$. Recall from Lemma \ref{L3} that the genealogical trees $T_j(n)$ and $T_{j'}(n)$ are then two random recursive trees with respective sizes $k$ and $k'$, which are further independent when $j\neq j'$. Thanks to Corollary \ref{C0}(i-ii) 
we get  
\begin{align*} &\EE\left(  \Delta(T_j(n)) \Ind_{{N}_{j'}(n)\geq K}  \Delta(T_{j'}(n)) \Ind_{{N}_{j'}(n)\geq K}\right)
\\&= \left\{\begin{matrix} 
\frac{1}{3}  \EE(N_j(n)\Ind_{{N}_j(n)\geq K}) &\text{ if }j=j'\,\\
0 &\text{ if }j\neq j'.
\end{matrix}\right.
\end{align*}

We have thus shown that 
$$\EE\left( \left | \sum_{k\geq K} \check S_k(n)\right|^2 \right) \leq \frac{m_2 }{3}\sum_{j=1}^{n} \EE(N_j(n)\Ind_{{N}_j(n)\geq K} ),$$
which yields our claim just as in the proof of Lemma \ref{L7}(ii). 
\end{proof}

The proof of Theorem \ref{T2} is now easily completed by combining Lemmas \ref{L8} and \ref{L9}.
Indeed, the  identity
$$\frac{pm_2}{2-p}+\sum_{k=2}^{\infty}\sigma^2_{k} = \frac{m_2}{3-2p}$$
is easily checked from \eqref{E:sumkbeta}.

\section{A stable central limit theorem}

The arguments for the proof of Theorem \ref{T2} when the step distribution $\mu$ has a finite second moment can be adapted to the case when $\mu$  belongs to some stable domain of attraction; for the sake of simplicity we focus on the situation without centering. Specifically, let  $(a_n)$ be a sequence of positive real numbers that is regularly varying with exponent $1/\alpha$ for some $\alpha\in(0,2)$, in the sense that $\lim_{n\to \infty} a_{\lfloor rn\rfloor}/a_n= r^{1/\alpha}$ for every $r>0$, and suppose that 
\begin{equation} \label{E:stable}
\lim_{n\to \infty} \frac{X_1+\cdots + X_n}{a_n} = Z \qquad \text{in distribution},
\end{equation} 
where $Z$ is some $\alpha$-stable random variable. We refer to Theorems 4 and 5 on p. 181-2 in \cite{GnKo} and Section 6 of Chapter 2 in \cite{IbLin}
for necessary and sufficient conditions for \eqref{E:stable} in terms of $\mu$ only. 
We write $\varphi_{\alpha}$ for the characteristic exponent of $Z$, viz.
$$\EE\left(\exp(\iu \theta Z)\right) = \exp(-\varphi_{\alpha}(\theta))\qquad \text{for all }\theta \in \RR;$$
recall that $\varphi_{\alpha}$ is homogeneous with exponent $\alpha$, i.e. 
$$\varphi_{\alpha}(\theta)= |\theta|^{\alpha} \varphi_{\alpha}(\textrm{sgn}(\theta))\qquad \text{for all }\theta\neq 0.$$

Recall the definition and properties  of the Eulerian numbers $\langle {}^{n}_{k}\rangle$  from Section 2, and also the Pochhammer notation 
$$(x)^{(k)} \coloneqq \frac{\Gamma(x+k)}{\Gamma(x)} =\prod_{j=0}^{k-1} (x+j), \qquad x>0, k\in\NN,$$ 
for the rising factorial, where  $\Gamma$ stands for the gamma function.
 We can now claim:

\begin{theorem}\label{T3}  Assume \eqref{E:stable}. For each $p\in(0,1)$, we have
$$\lim_{n\to \infty} \frac{\check S(n)}{a_n} = \check Z \qquad \text{in distribution},$$
where $\check Z$ is an $\alpha$-stable random variable with  characteristic exponent $\check \varphi_{\alpha}$ given by
$$\check \varphi_{\alpha}(\theta) =  \frac{p  }{(1-p)} \sum_{k=1}^{\infty}  \sum_{\ell=0}^{k-1}\frac{ \varphi_{\alpha}((k-2\ell)\theta)}{ (1+1/(1-p))^{(k)}} \euler{{k-1}}{{\ell-1}}, \qquad \theta \in \RR.$$
\end{theorem}

The  proof of Theorem \ref{T3} relies on a refinement of Simon's result (Lemma \ref{L2}) to the asymptotic  frequencies of genealogical trees induced by the reinforcement algorithm \eqref{E:reinf}. We denote by ${{\mathcal T}^{\uparrow}}$ the set of increasing trees (of arbitrary finite size), and for any $\tau\in {\mathcal T}^{\uparrow}$, we write $|\tau|$ for its the size (number of vertices) and $\Delta(\tau)$
for the difference between its numbers of even vertices and of odd vertices. Refining \eqref{E:nudef}, we  also  define
$$\nu_{\tau}(n) \coloneqq \sum_{j= 1}^{{\mathrm i}(n)} \Ind_{\{T_j(n)  =\tau\}},  \qquad \tau\in {\mathcal T}^{\uparrow}.$$
\begin{lemma}\label{L2+}  We have
$$\sum_{\tau\in{{\mathcal T}^{\uparrow}}} 
\frac{ |\tau|+ |\Delta(\tau)|^2}{(1+1/(1-p))^{(|\tau|)}}= \frac{4p}{3(1-p)}, 
$$
 and there is the convergence in probability
$$\lim_{n\to \infty} \sum_{\tau\in{{\mathcal T}^{\uparrow}}} (|\tau|+ |\Delta(\tau)|^2)\left| \frac{\nu_{\tau}(n)}{n} - 
\frac{  p}{(1-p) (1+1/(1-p))^{(|\tau|)}} \right| =0.$$
\end{lemma} 

\begin{proof} 
 We start by claiming that for every $k\geq 1$, and every tree $\tau\in {{\mathcal T}^{\uparrow}}$ with size $|\tau|=k$,   
  we have
\begin{equation}\label{E:excor}
\lim_{n\to \infty} \frac{\nu_{\tau}(n)}{n} =
\frac{ p}{(1-p)  (1+1/(1-p))^{(k)}} \qquad\text{in probability.}
\end{equation}
 Indeed, the distribution of the random recursive tree $\mathbb T_k$ of size $k$ is  the uniform probability measure on the set 
of increasing trees with size $k$, which has $(k-1)!$ elements. We deduce from Lemma \ref{L3} and the law of large numbers that 
$$\nu_{\tau}(n) \sim \nu_k(n)/(k-1)!.$$ 
The claim \eqref{E:excor} now follows from Lemma \ref{L2} and the identity 
$$ {\mathrm B}(k,1+1/(1-p)) =  \frac{ (k-1)! }{ (1+1/(1-p))^{(k)}}.$$

We now have to prove that \eqref{E:excor} holds in $L^1(|\tau|+ |\Delta(\tau)|^2,{\mathcal T}^{\uparrow})$.
On the one hand, one has obviously for every $n\geq 1$
$$\sum_{\tau\in{{\mathcal T}^{\uparrow}}} |\tau|\nu_{\tau}(n)= n.$$
On the other hand, there are $(k-1)!$ increasing trees with size $k$ and hence
$$\frac{p}{1-p} \sum_{\tau\in{{\mathcal T}^{\uparrow}}} \frac{|\tau|}{(1+1/(1-p))^{(|\tau|)}}= \frac{p}{1-p} \sum_{k=1}^{\infty} k{\mathrm B}(k,1+1/(1-p))=1,$$
where the second equality is \eqref{E:sumkbeta}. 
We deduce from  Scheff\'e's Lemma  and \eqref{E:excor} that there is the convergence in probability
$$\lim_{n\to \infty} \sum_{\tau\in{{\mathcal T}^{\uparrow}}}|\tau|\left| \frac{\nu_{\tau}(n)}{n} - 
\frac{  p}{(1-p) (1+1/(1-p))^{(|\tau|)}} \right| =0.$$

Similarly, we deduce from Corollary \ref{C0}(ii) and Lemma \ref{L3} that, for every $n\geq0$
\begin{align*}
\EE\left( \sum_{\tau\in{{\mathcal T}^{\uparrow}}} \Delta(\tau)^2 \nu_{\tau}(n)\right) & =  \EE\left(\sum_{j=1}^{{\mathrm i}(n)} \Delta(T_j(n))^2\right)
 = \frac{1}{3} \EE\left(\sum_{j=1}^{{\mathrm i}(n)} |T_j(n)|\right) = n/3,
\end{align*}
and further, since there are $(k-1)!$ increasing trees with size $k$ and $\mathbb T_k$ has the uniform distribution on the set of such trees, 
\begin{align*}\frac{p}{1-p} \sum_{\tau\in{{\mathcal T}^{\uparrow}}} \frac{\Delta(\tau)^2}{(1+1/(1-p))^{(|\tau|)}}&= \frac{p}{1-p} \sum_{k=1}^{\infty} \EE(\Delta(\mathbb T_k)^2){\mathrm B}(k,1+1/(1-p))\\
&= \frac{p}{1-p} \sum_{k=1}^{\infty}  \frac{k}{3}{\mathrm B}(k,1+1/(1-p))=\frac{1}{3}.
\end{align*}
We conclude again from Scheff\'e's Lemma that
$$\lim_{n\to \infty} \sum_{\tau\in{{\mathcal T}^{\uparrow}}}\Delta(\tau)^2\left| \frac{\nu_{\tau}(n)}{pn} - 
\frac{  p}{(1-p) (1+1/(1-p))^{(|\tau|)}} \right| =0,$$
and the proof is complete. 
\end{proof}

We now establish Theorem \ref{T3}.
\begin{proof}[Proof of Theorem \ref{T3}]
We denote the characteristic function of $\mu$ by 
$$\Phi(\theta)=\int_{\RR} \e^{\mathrm i \theta x} \mu(\mathrm d x) \qquad \text{for }\theta\in \RR.$$
Fix $r>0$ small enough so that $|1-\Phi(\theta)|<1$ whenever $|\theta|\leq r$, and then define the characteristic exponent $\varphi:[-r,r]\to \CC$
as the continuous determination of the logarithm of $\Phi$ on $[-r,r]$. In words, $\varphi$ is the 
unique continuous function on $[-r,r]$  with $\varphi(0)=0$ and such that $\Phi(\theta)= \exp(-\varphi(\theta))$ for all $\theta\in[-r,r]$. For definitiveness, we further set
$\varphi(\theta)=0$ whenever $|\theta|>r$.

Next, observe from the Markov's inequality that for any $1<\beta<2\wedge(1-p)^{-1}$ and any $a>0$
$$\PP\left( \exists j\leq  {\mathrm i}(n): |\Delta(T_j(n))|\geq a\sqrt n\right) \leq a^{-2\beta} n^{-\beta} \EE\left(\sum_{j=1}^{{\mathrm i}(n)} |\Delta(T_j(n))|^{2\beta}\right)
,$$
so that, thanks to Lemma \ref{L6},
$$\lim_{n\to \infty} \frac{1}{\sqrt n} \max_{1\leq j \leq {\mathrm i}(n)} |\Delta(T_j(n))|=0\qquad \text{ in probability.}$$
In particular, since the sequence $(a_n)$ is regularly varying with exponent $1/\alpha >1/2$, for every $\theta\in \RR$, 
the events 
$$\Lambda(n,\theta)\coloneqq \{ |\theta \Delta(T_j(n))/a_n|< r\text{ for all } j=1, \ldots {\mathrm i}(n) \}, \qquad n\geq 1$$
occur with high probability as $n\to \infty$, in the sense that $\lim_{n\to \infty} \PP(\Lambda(n,\theta))=1$. 

We then deduce from \eqref{E:repres} and the fact that the variables $X_j$ are i.i.d. with law $\mu$ that for every $\theta\in \RR$,
$$\EE(\exp(\mathrm i \theta \check S(n)/a_n)\Ind_{\Lambda(n,\theta)}) =
\EE\left(\exp \left( - \frac{1}{n} \sum_{j=1}^{{\mathrm i}(n)} n\varphi \left( \theta a_n^{-1} \Delta(T_j(n))\right)\right)\Ind_{\Lambda(n,\theta)} \right).$$
We then write, in the notation of Lemma \ref{L2+},
$$\frac{1}{n} \sum_{j=1}^{{\mathrm i}(n)} n\varphi \left( \theta a_n^{-1} \Delta(T_j(n))\right)= \sum_{\tau\in{\mathcal T}^{\uparrow}}n\varphi \left( \theta a_n^{-1} \Delta(\tau)\right)  \frac{\nu_{\tau}(n)}{n}.$$

Recall that we assume   \eqref{E:stable}.  According to Theorem 2.6.5 in Ibragimov and Linnik \cite{IbLin}, $\varphi$ is regularly varying at $0$ with exponent $\alpha$, and since  $\alpha<2$, the Potter bounds (see Theorem 1.5.6 in \cite{BGT}) show that for some constant $C$: 
\begin{equation}\label{E:Potter}n|\varphi \left( \theta a_n^{-1} \Delta(\tau)\right)|\leq C |\theta \Delta(\tau)|^2.
\end{equation}
We deduce from Lemma \ref{L2+} that for every fixed $\theta\in \RR$,  there is the convergence in probability
$$\lim_{n\to \infty} \sum_{\tau\in{\mathcal T}^{\uparrow}}n|\varphi \left( \theta a_n^{-1} \Delta(\tau)\right) | \left| \frac{\nu_{\tau}(n)}{n} -
\frac{  p}{(1-p) (1+1/(1-p))^{(|\tau|)}} \right|=0 .$$
Furthermore, still from Theorem 2.6.5 in Ibragimov and Linnik \cite{IbLin}, we have
$$
\lim_{n\to \infty}n\varphi(\theta /a_n) = \varphi_{\alpha}(\theta), \qquad \text{for every }\theta\in\RR,
$$
and we deduce by dominated convergence, using  Lemma \ref{L2+} and \eqref{E:Potter}, that
$$\lim_{n\to \infty} \sum_{\tau\in{\mathcal T}^{\uparrow}}|n \varphi \left( \theta a_n^{-1} \Delta(\tau)\right) -\varphi_{\alpha}(\theta \Delta(\tau))|  
\frac{ p}{(1-p) (1+1/(1-p))^{(|\tau|)}} =0 .$$

Putting the pieces together, we have shown that
$$\lim_{n\to \infty} \EE(\exp(\mathrm i \theta \check S(n)/a_n)) = \exp\left( - 
 \sum_{\tau\in{{\mathcal T}^{\uparrow}}}  \varphi_{\alpha}(\theta \Delta(\tau))
\frac{  p}{(1-p) (1+1/(1-p))^{(|\tau|)}} \right) . 
$$
It only remains to check that the right-hand side above agrees  with the formula of the statement. This follows from Lemma \ref{L1} and the fact that for every $k\geq 1$,  $\mathbb T_k$ has the uniform distribution on $\{\tau\in{{\mathcal T}^{\uparrow}}: |\tau|=k\}$. 
\end{proof} 

\eject

 \bibliography{NoiseR.bib}

\begin{thebibliography}{10}

\bibitem{BHZ}
{\sc Bai, Z.~D., Hu, F., and Zhang, L.-X.}
\newblock Gaussian approximation theorems for urn models and their
  applications.
\newblock {\em Ann. Appl. Probab. 12}, 4 (2002), 1149--1173.

\bibitem{Baur}
{\sc Baur, E.}
\newblock On a class of random walks with reinforced memory.
\newblock {\em Journal of Statistical Physics 181\/} (2020), 772--802.

\bibitem{BBsurvey}
{\sc Baur, E., and Bertoin, J.}
\newblock Cutting edges at random in large recursive trees.
\newblock In {\em Stochastic analysis and applications 2014}, vol.~100 of {\em
  Springer Proc. Math. Stat.} Springer, Cham, 2014, pp.~51--76.

\bibitem{BaurBer}
{\sc Baur, E., and Bertoin, J.}
\newblock Elephant random walks and their connection to {P}\'olya-type urns.
\newblock {\em Phys. Rev. E 94\/} (Nov 2016), 052134.

\bibitem{Bercu}
{\sc Bercu, B.}
\newblock A martingale approach for the elephant random walk.
\newblock {\em J. Phys. A 51}, 1 (2018), 015201, 16.

\bibitem{BerLau2}
{\sc Bercu, B., and Laulin, L.}
\newblock On the center of mass of the elephant random walk.
\newblock {\em Stochastic Processes and their Applications 133\/} (2021), 111
  -- 128.

\bibitem{Marco2}
{\sc Bertenghi, M.}
\newblock Asymptotic normality of superdiffusive step-reinforced random walks.
\newblock arXiv:2101.00906.

\bibitem{Marco}
{\sc Bertenghi, M.}
\newblock Functional limit theorems for the multi-dimensional elephant random
  walk.
\newblock {\em Stoch. Models 38}, 1 (2022), 37--50.

\bibitem{NRLP}
{\sc Bertoin, J.}
\newblock Noise reinforcement for {L}\'evy processes.
\newblock {\em Ann. Inst. H. Poincar\'e Probab. Statist. 56}, 3 (2020),
  2236--2252.

\bibitem{Bescal}
{\sc Bertoin, J.}
\newblock Scaling exponents of step-reinforced random walks.
\newblock {\em Probab. Theory Relat. Fields 179\/} (2021), 295--315.

\bibitem{UNRBM}
{\sc Bertoin, J.}
\newblock Universality of noise reinforced {B}rownian motions.
\newblock In {\em In and out of equilibrium 3. {C}elebrating {V}ladas
  {S}idoravicius}, vol.~77 of {\em Progr. Probab.} Birkh\"{a}user/Springer,
  Cham, [2021] \copyright 2021, pp.~147--161.

\bibitem{BGT}
{\sc Bingham, N.~H., Goldie, C.~M., and Teugels, J.~L.}
\newblock {\em Regular variation}, vol.~27 of {\em Encyclopedia of Mathematics
  and its Applications}.
\newblock Cambridge University Press, Cambridge, 1987.

\bibitem{BRST}
{\sc Bollob\'{a}s, B., Riordan, O., Spencer, J., and Tusn\'{a}dy, G.}
\newblock The degree sequence of a scale-free random graph process.
\newblock {\em Random Structures Algorithms 18}, 3 (2001), 279--290.

\bibitem{Buart}
{\sc Businger, S.}
\newblock The shark random swim ({L}\'{e}vy flight with memory).
\newblock {\em J. Stat. Phys. 172}, 3 (2018), 701--717.

\bibitem{ColGavSch1}
{\sc Coletti, C.~F., Gava, R., and Sch\"{u}tz, G.~M.}
\newblock Central limit theorem and related results for the elephant random
  walk.
\newblock {\em J. Math. Phys. 58}, 5 (2017), 053303, 8.

\bibitem{ColGavSch2}
{\sc Coletti, C.~F., Gava, R., and Sch\"{u}tz, G.~M.}
\newblock A strong invariance principle for the elephant random walk.
\newblock {\em J. Stat. Mech. Theory Exp.}, 12 (2017), 123207, 8.

\bibitem{Drmota}
{\sc Drmota, M.}
\newblock {\em Random trees}.
\newblock Springer Wien, NewYork, Vienna, 2009.
\newblock An interplay between combinatorics and probability.

\bibitem{Durrett}
{\sc Durrett, R.}
\newblock {\em Random graph dynamics}, vol.~20 of {\em Cambridge Series in
  Statistical and Probabilistic Mathematics}.
\newblock Cambridge University Press, Cambridge, 2007.

\bibitem{GnKo}
{\sc Gnedenko, B.~V., and Kolmogorov, A.~N.}
\newblock {\em Limit distributions for sums of independent random variables}.
\newblock Translated from the Russian, annotated, and revised by K. L. Chung.
  With appendices by J. L. Doob and P. L. Hsu. Revised edition. Addison-Wesley
  Publishing Co., Reading, Mass.-London-Don Mills., Ont., 1968.

\bibitem{GutStadt}
{\sc Gut, A., and Stadtm\"{u}ller, U.}
\newblock The number of zeros in elephant random walks with delays.
\newblock {\em Statist. Probab. Lett. 174\/} (2021), Paper No. 109112, 9.

\bibitem{IbLin}
{\sc Ibragimov, I.~A., and Linnik, Y.~V.}
\newblock {\em Independent and stationary sequences of random variables}.
\newblock Wolters-Noordhoff Publishing, Groningen, 1971.
\newblock With a supplementary chapter by I. A. Ibragimov and V. V. Petrov,
  Translation from the Russian edited by J. F. C. Kingman.

\bibitem{KuMa}
{\sc Kubota, N., and Takei, M.}
\newblock Gaussian fluctuation for superdiffusive elephant random walks.
\newblock {\em J. Stat. Phys. 177}, 6 (2019), 1157--1171.

\bibitem{Kur}
{\sc K\"{u}rsten, R.}
\newblock Random recursive trees and the elephant random walk.
\newblock {\em Phys. Rev. E 93}, 3 (2016), 032111, 11.

\bibitem{Mahmoud}
{\sc Mahmoud, H.~M.}
\newblock {\em P\'{o}lya urn models}.
\newblock Texts in Statistical Science Series. CRC Press, Boca Raton, FL, 2009.

\bibitem{MeirMoon}
{\sc Meir, A., and Moon, J.}
\newblock Cutting down recursive trees.
\newblock {\em Mathematical Biosciences 21}, 3 (1974), 173 -- 181.

\bibitem{Mitak}
{\sc Miyazaki, T., and Takei, M.}
\newblock Limit theorems for the `laziest' minimal random walk model of
  elephant type.
\newblock {\em J. Stat. Phys. 181}, 2 (2020), 587--602.

\bibitem{NaHeyde}
{\sc Najock, D., and Heyde, C.~C.}
\newblock On the number of terminal vertices in certain random trees with an
  application to stemma construction in philology.
\newblock {\em J. Appl. Probab. 19}, 3 (1982), 675--680.

\bibitem{PPS}
{\sc Pachon, A., Polito, F., and Sacerdote, L.}
\newblock Random graphs associated to some discrete and continuous time
  preferential attachment models.
\newblock {\em J. Stat. Phys. 162}, 6 (2016), 1608--1638.

\bibitem{Pem}
{\sc Pemantle, R.}
\newblock A survey of random processes with reinforcement.
\newblock {\em Probab. Surveys 4\/} (2007), 1--79.

\bibitem{Peter}
{\sc Petersen, T.~K.}
\newblock {\em Eulerian numbers}.
\newblock Birkh\"{a}user Advanced Texts: Basler Lehrb\"{u}cher. [Birkh\"{a}user
  Advanced Texts: Basel Textbooks]. Birkh\"{a}user/Springer, New York, 2015.
\newblock With a foreword by Richard Stanley.

\bibitem{SchTr}
{\sc Sch\"utz, G.~M., and Trimper, S.}
\newblock Elephants can always remember: Exact long-range memory effects in a
  non-{M}arkovian random walk.
\newblock {\em Phys. Rev. E 70\/} (Oct 2004), 045101.

\bibitem{Simon}
{\sc Simon, H.~A.}
\newblock On a class of skew distribution functions.
\newblock {\em Biometrika 42}, 3/4 (1955), 425--440.

\bibitem{Stanley}
{\sc Stanley, R.~P.}
\newblock {\em Enumerative combinatorics. {V}ol. 1}, vol.~49 of {\em Cambridge
  Studies in Advanced Mathematics}.
\newblock Cambridge University Press, Cambridge, 1997.
\newblock With a foreword by Gian-Carlo Rota, Corrected reprint of the 1986
  original.

\bibitem{Tanny}
{\sc Tanny, S.}
\newblock A probabilistic interpretation of {E}ulerian numbers.
\newblock {\em Duke Math. J. 40}, 4 (12 1973), 717--722.

\end{thebibliography}

\end{document}